\documentclass[preprint,12pt]{elsarticle}
\usepackage{amsmath, amssymb, amsthm}

\theoremstyle{plain}
\newtheorem{theorem}{Theorem}
\newtheorem{corollary}[theorem]{Corollary}
\newtheorem{lemma}[theorem]{Lemma}
\newtheorem{proposition}[theorem]{Proposition}
\newtheorem{problem}[theorem]{Problem}
\newtheorem{conjecture}[theorem]{Conjecture}

\theoremstyle{definition}
\newtheorem{definition}[theorem]{Definition}
\newtheorem{example}[theorem]{Example}
\newtheorem{remark}[theorem]{Remark}
\newtheorem{question}[theorem]{Question}
\begin{document}

\begin{frontmatter}

\title{Strongly ultrametric preserving functions}
\author{Oleksiy Dovgoshey}
\ead{oleksiy.dovgoshey@gmail.com, oleksiy.dovgoshey@utu.fi}

\affiliation{organization={Institute of Applied Mathematics and Mechanics of the NAS of Ukraine}, 
city={Sloviansk},
country={Ukraine}}

\affiliation{organization={Department of Mathematics and Statistics, University of Turku},
state={Turku},
country={Finland}}

\begin{abstract}
An ultrametric preserving function $f$ is said to be strongly ultrametric preserving if ultrametrics $d$ and $f \circ d$ define the same topology on $X$ for each ultrametric space $(X,d)$. The set of all strongly ultrametric preserving functions is characterized by several distinctive features. In particular, it is shown that an ultrametric preserving $f$ belongs to this set iff $f$ preserves the property to be compact.
\end{abstract}


\end{frontmatter}

\section{Introduction}

The following concept was introduced by Jacek Jachymski and Filip Turobo\'{s} in \cite{JTRRACEFNSAMR2020}.

\begin{definition}
Let $\mathbf{A}_1$ and $\mathbf{A}_2$ be two classes of semimetric spaces. We say that a function $f \colon [0, \infty) \to [0, \infty)$ is $\mathbf{A}_1$- $\mathbf{A}_2$-preserving if $(X, f \circ d)$ belongs to $\mathbf{A}_2$ for each $(X,d) \in \mathbf{A}_1$.
\end{definition}

This definition allows us to formulate

\begin{problem}\label{problem_2}
Let $\mathbf{A}_1$ and $\mathbf{A}_2$ be two classes of semimetric spaces. Find the characteristic properties of $\mathbf{A}_1$ - $\mathbf{A}_2$ preserving functions.
\end{problem}

We denote by $\mathbf{P}_{\mathbf{A}_1, \mathbf{A}_2}$ the set of all $\mathbf{A}_1$ - $\mathbf{A}_2$ preserving functions, and, for simplicity, write
$$
\mathbf{P}_{\mathbf{A}} := \mathbf{P}_{\mathbf{A}, \mathbf{A}}.
$$

\begin{definition}
Let $(X,d)$ be a semimetric space. Then

$ (X,d) \in \mathbf{CU}$ iff $(X,d)$ is a compact ultrametric space;

$ (X,d) \in \mathbf{M}$ iff $(X,d)$ is a metric space;

$ (X,d) \in \mathbf{NUDU}$ iff $(X,d)$ is not an uniformly discrete ultrametric space;

$ (X,d) \in \mathbf{TBU}$ iff $(X,d)$ is a totally bounded ultrametric space;

$ (X,d) \in \mathbf{U}$ iff $(X,d)$ is an ultrametric space;

$ (X,d) \in \mathbf{U}3$ iff $(X,d)$ is an ultrametric space and $|X| = 3$.
\end{definition}

In what follows we also write $\mathbf{PT}$ for the set of all strongly ultrametric preserving functions.

The main purpose of the paper is to find solutions of Problem~\ref{problem_2} for:
$$
\mathbf{A}_1 = \mathbf{A}_2 = \mathbf{CU};
\quad
\mathbf{A}_1 = \mathbf{CU} \quad \textrm{and} \quad \mathbf{A}_2= \mathbf{TBU}; \quad \mathbf{A}_1 = \mathbf{A}_2 = \mathbf{NUDU}.
$$

The paper is organized as follows.
The next section contains the sequential characterizations of $\mathbf{CU}$, $\mathbf{TBU}$, $\mathbf{NUDU}$ and, moreover, some definitions and facts related to ultrametric preserving functions. The main results are formulated and proven in Section~3.

Theorem~\ref{th22} contains a general sufficient condition under which the inclusion $\mathbf{P}_{\mathbf{A}} \subseteq \mathbf{P}_{\mathbf{U}}$ holds for $\mathbf{A} \subseteq \mathbf{U}$. Theorem~\ref{th18} shows that
$$\mathbf{PT}= \mathbf{P}_{\mathbf{CU}}= \mathbf{P}_{\mathbf{TBU}} =\mathbf{P}_{\mathbf{CU}, \mathbf{TBU}}= \mathbf{P}_{\mathbf{NUDU}}.
$$

A dual form of Theorem~\ref{th18} is given in Proposition~\ref{prop_27}.

Fourth section of the paper contains Proposition~\ref{prop_22} characterizing the ultrametric preserving functions via a functional equation. Corollary~\ref{cor_34} shows that an amenable $f$ is increasing and subadditive iff $f \in \mathbf{P}_{\mathbf{M}} \cap \mathbf{P}_{\mathbf{U}}$.

Two conjectures connected with boundedly compact ultrametric spaces and locally finite ones are formulated in the final Section~5.

\section{Preliminaries}

Let us start from the classical notion of metric space introduced by Maurice
Fr\'{e}chet in his thesis \cite{Fre1906RdCMdP}.

A \emph{metric} on a set $X$ is a function $d: X \times X \to [0, \infty)$ such that for all $x, y, z \in X$:
\begin{itemize}
\item[$(i)$] $d(x, y) \geqslant 0$ with equality if and only if $x = y$, the \emph{positivity property};

\item[$(ii)$] $d(x, y) = d(y, x)$, the \emph{symmetry property};

\item[$(iii)$] $d(x, y) \leqslant d(x, z) + d(z, y)$, the \emph{triangle inequality}.
\end{itemize}

An useful generalization of the notion of metric space is the concept of
semimetric space.

\begin{definition}\label{def_1.1}
Let $X$ be a set and let $d : X \times X \to [0, \infty)$ be a symmetric function. The function $d$ is a \emph{semimetric} on $ X$ if it satisfies the positivity property.
\end{definition}

If $d$ is a semimetric on $X$, we say that $(X, d)$ is a \emph{semimetric space}. A metric space $(X,d)$ is \emph{ultrametric} if the \emph{strong triangle inequality}
$$
d(x,y) \leqslant \max \{ d(x,z), d(z,y) \}
$$
holds for all $x,y,z \in X$.

Let $(X,d)$ be a metric space. A sequence $(x_n)_{n\in \mathbb{N}}$ of points of a metric space $(X, d)$ is said to \emph{converge} to a point $p \in X$ if
\begin{equation*}
\lim_{n \to \infty} d(x_n, p) =0.
\end{equation*}

Recall also that a sequence $(x_n)_{n\in\mathbb{N}}$ of points of $X$ is a \emph{Cauchy sequence} in a metric space $(X,d)$ iff
$$
\lim_{\substack{n \to \infty \\ m \to \infty}} d(x_n, x_m) = 0.
$$

\begin{proposition}\label{pr_BW}
A subset $A$ of a metric space is
compact if and only if every infinite sequence of points of $A$ contains a subsequence which converges to a point of $A$.
\end{proposition}

This result and a long list of distinct criteria of compactness can be found, for example, in \cite[p. 206]{Sea2007}.

We will also use the following proposition.

\begin{proposition}\label{pr_7}
A subset $A$ of a metric space $(X, d)$ is totally bounded if and only if every infinite sequence of points of $A$ contains a Cauchy subsequence.
\end{proposition}

(See, for example, Theorem~7.8.2 in \cite{Sea2007}.)

Each compact space is totally bounded.
In particular,
\begin{equation}\label{s2_eq3}
\mathbf{CU} \subseteq \mathbf{TBU}.
\end{equation}
holds.

Some new criteria for compactness and total boundedness of ultrametric spaces were recently obtained in \cite{DS2021TA, DK2021AC, BDK2021TAG}.

Compact ultrametric spaces and totally bounded ones are the main objects of study in the present paper, but we also need the concept of uniformly discrete ultrametric space.

\begin{definition}\label{def_1.5}
A metric space $(X, d)$ is said to be \emph{uniformly discrete} if there is $\varepsilon \in (0, \infty)$ such that
$$
d(x,y) > \varepsilon
$$
whenever $x$ and $y$ are different points of $ X$.
\end{definition}

The following result directly follows from Definition~\ref{def_1.5}.

\begin{proposition}\label{prop_11}
Let $(X,d)$ be an ultrametric space. Then $(X,d)$ is a not uniformly discrete space iff there exist two sequences $(x_n)_{n \in \mathbb{N}}$ and $(y_n)_{n \in \mathbb{N}}$ of points of $X$ such that
$$
\lim_{n\to \infty} d(x_n, y_n) =0
$$
and $x_n \neq y_n$ for every $n \in \mathbb{N}$.
\end{proposition}

\begin{remark}\label{rem_12}
To the best of the author's knowledge, Proposition~\ref{prop_11} is new.
\end{remark}

We now turn our attention to functions that preserve metrics.
We will say that a function $f \colon [0, \infty) \to [0, \infty)$ is \emph{metric preserving} iff $f \in \mathbf{P}_{\mathbf{M}}$. Similarly $f \colon [0, \infty) \to [0, \infty)$ is \emph{ultrametric preserving} iff $f \in \mathbf{P}_{\mathbf{U}}$.

It is straightforward to see that each metric preserving $f\colon [0, \infty) \rightarrow [0, \infty)$ must be \emph{amenable}, i.e., \(f(0) = 0\) and \(f(x) > 0\) for every \(x \in (0, \infty) \).

\begin{example}\label{ex_11}
Let $f\colon [0, \infty) \rightarrow [0, \infty)$ be amenable.
If $f$ is increasing and subaddive, then $f \in \mathbf{P}_{\mathbf{M}}$ (see Theorem~1 in \cite[p.~5]{Dobos1998})).
It was proved by Dobo\v{s} \cite{D1996PAMS} that the extended Cantor function $\widehat{G}$ is subadditive,
$$
\widehat{G} (x+y) \leqslant \widehat{G} (x) + \widehat{G} (y)
$$
for all $x,y \in [0, \infty)$. Since $\widehat{G}$ is increasing, $\widehat{G} \in \mathbf{P}_{\mathbf{M}}$ holds.
\end{example}

\begin{remark}\label{rem_13}
The standard Cantor ternary function $G$ is defined on the set $[0,1]$ and, consequently, $G \notin \mathbf{P}_{\mathbf{M}}$. Using the Cantor function $G$ we can introduce $\widehat{G}$ as
$$
\widehat{G}(x) := \left\{
\begin{array}{ll}
G(x) & \quad \hbox{if}\quad x \in [0,1], \\
1 & \quad \hbox{if} \quad x \in (1, \infty).
\end{array}
\right.
$$
A systematic survey of properties of $G$ can be found in \cite{DMRV2006EM}.
\end{remark}

A function $f \in \mathbf{P}_{\mathbf{M}}$ is said to be \emph{strongly metric preserving} (see \cite[p.~109]{MDED}) if the metrics $d$ and $f \circ d$ define the same topology on $X$ for every $(X,d) \in
\mathbf{M}$.

Let us define now the set $\mathbf{PT}$ of all strongly ultrametric preserving functions.

\begin{definition}\label{def11}
An ultrametric preserving function $f$ is said to be \emph{strongly ultrametric preserving} if $d$ and $f \circ d$ define the same topology for every ultrametric space $(X,d)$.
\end{definition}

The concept of metric preserving functions can be traced back to Wilson \cite{Wilson1935}. Similar problems were
considered by Blumenthal in \cite{Blumenthal1936}. The theory of metric preserving functions was developed by Bors\'{\i}k, Dobo\v{s}, Piotrowski, Vallin and other mathematicians \cite{BD1981MS, Borsik1988, Dobos1996, Dobos1994, Dobos1996a, Dobos1997, V1997RAE, V1998AMUC, V1998IJMMS, Vallin2000, PT2014FPTA, V2002TMMP}.
See also lectures by Dobo\v{s} \cite{Dobos1998}, an introductory paper by Corazza \cite{Corazza1999}.

The study of ultrametric preserving functions begun by P.~Pongsriiam and I.~Termwuttipong in 2014~\cite{PTAbAppAn2014} and was continued in \cite{Dov2020MS, VD2021MS}. The following theorem, from \cite{PTAbAppAn2014} gives us an easy and complete description of the set $\mathbf{P}_{\mathbf{U}}$.

\begin{theorem}\label{t2.4}
A function \( f\colon [0, \infty) \rightarrow [0, \infty)\) is ultrametric preserving if and only if \(f\) is increasing and amenable.
\end{theorem}

\begin{example}\label{ex_14}
The extended Cantor function $\widehat{G}$ is amenable and increasing. Hence $\widehat{G} \in \mathbf{P}_{\mathbf{U}}$ holds.
\end{example}

Theorem~\ref{t2.4} was generalized in \cite{Dov2019a} to the special case of the so-called ultrametric distances. These distances were introduced by S.~Priess-Crampe and P.~Ribenboim in 1993 \cite{PR1993AMSUH} and studied in \cite{PR1996AMSUH, PR1997AMSUH, Rib1996PMH, Rib2009JoA}. The functions preserving $p$-adic metrics and some other classes of ultrametrics were first considered in \cite{VD2021MS}.

\begin{remark}\label{rem_1.8}
The metric preserving functions can be considered as a special case of metric products (= metric preserving functions of several variables). See, for example, \cite{BD1981, BFS2003BazAaG, DPK2014MS, FS2002, HMCM1991, Kaz2021CoPS}. It is interesting to note that an important special class of ultrametric preserving functions of two variables was first considered in 2009~\cite{DM2009}.
\end{remark}


\section{Main results and some lemmas}

Let us start from the following useful fact.

\begin{proposition}\label{lem_3.1}
Let $\mathbf{A}_1$, $\mathbf{A}_2$, $\mathbf{A}_3$, $\mathbf{A}_4$ be classes of semimetric spaces. Then the following statements hold:
\begin{itemize}
\item[$(i)$] If $\mathbf{A}_1 \subseteq \mathbf{A}_2$, then
$\mathbf{P}_{\mathbf{A}_1, \mathbf{A}_3} \supseteq \mathbf{P}_{\mathbf{A}_2, \mathbf{A}_3}$;
\item[$(ii)$]
If $\mathbf{A}_3 \subseteq \mathbf{A}_4$, then
$
\mathbf{P}_{\mathbf{A}_2, \mathbf{A}_4} \supseteq \mathbf{P}_{\mathbf{A}_2, \mathbf{A}_3}.
$
\end{itemize}
\end{proposition}

\begin{proof}
It follows directly from Definition~1.
\end{proof}

\begin{corollary}\label{cor_3.2}
The inclusions
\begin{equation}\label{cor3.2_3.3}
\mathbf{P}_{\mathbf{U3}, \mathbf{U}} \supseteq \mathbf{P}_{\mathbf{CU}, \mathbf{TBU}}, \quad \mathbf{P}_{\mathbf{U3}, \mathbf{U}} \supseteq \mathbf{P}_{\mathbf{CU}} \quad \textit{and} \quad \mathbf{P}_{\mathbf{U3}, \mathbf{U}}\supseteq \mathbf{P}_{\mathbf{TBU}}
\end{equation}
hold.
\end{corollary}

\begin{proof}
It follows from
$$
\mathbf{U3} \subseteq \mathbf{CU} \subseteq \mathbf{TBU} \subseteq \mathbf{U}
$$
by Proposition~\ref{lem_3.1}.
\end{proof}

\begin{lemma}\label{lem_15}
The equalities
\begin{equation}\label{lem17_eq1}
\mathbf{P}_{\mathbf{U3}, \mathbf{U}}= \mathbf{P}_{\mathbf{U}},
\end{equation}
and
\begin{equation}\label{lem17_eq3}
\mathbf{P}_{\mathbf{U3}, \mathbf{U}}= \mathbf{P}_{\mathbf{U3}}
\end{equation}
hold.
\end{lemma}

\begin{proof}
The trivial inclusion $\mathbf{U3} \subseteq {\mathbf{U}}$
and Proposition~\ref{lem_3.1} imply
$$
\mathbf{P}_{\mathbf{U3}, \mathbf{U}} \supseteq \mathbf{P}_{\mathbf{U}}.
$$
Thus \eqref{lem17_eq1} holds iff
\begin{equation}\label{lem17_eq2}
\mathbf{P}_{\mathbf{U3}, \mathbf{U}} \subseteq \mathbf{P}_{\mathbf{U}}.
\end{equation}

Suppose contrary that there is $f^{\ast} \in \mathbf{P}_{\mathbf{U3}, \mathbf{U}}$ such that
\begin{equation}\label{eq_3.2}
f^{\ast}\notin \mathbf{P}_{\mathbf{U}}.
\end{equation}
Let $(X,d)$, $X \colon= \{x_1, x_2, x_3\}$, belong to $\mathbf{U3}$.
The membership relation $f^{\ast} \in \mathbf{P}_{\mathbf{U3}, \mathbf{U}}$ implies
\begin{equation}\label{eq_3.2_2}
f^{\ast} (0)=f^{\ast} ( d(x_1, x_1) )=0.
\end{equation}
Moreover, by Theorem~\ref{t2.4}, every increasing and amenable function belongs to $\mathbf{P}_{\mathbf{U}}$.
Thus \eqref{eq_3.2} holds iff either there is $c>0$ such that
$$
f^{\ast}(c)=0
$$
or $f^{\ast}(x)>0$ holds for every $x>0$ but there are $c_1, c_2 \in (0, \infty)$ satisfying the inequalities
\begin{equation}\label{eq_3.3}
0< c_1 < c_2 \quad \textrm{and} \quad 0< f^{\ast}(c_2) < f^{\ast}(c_1).
\end{equation}

Let us consider the equilateral triangle $X = \{ x_1, x_2, x_3\}$ in which all sides have the length $c$,
\begin{equation}\label{eq_3.4}
d(x_1,x_2) = d(x_2, x_3) = d(x_3, x_1) =c.
\end{equation}
Then $(X,d) \in \mathbf{U3}$ and consequently $f^{\ast} \in \mathbf{P}_{\mathbf{U3}, \mathbf{U}}$ implies
\begin{equation}\label{eq_3.5}
(X,f^{\ast} \circ d) \in {\mathbf{U}},
\end{equation}
but using \eqref{eq_3.2_2} and \eqref{eq_3.4} we see that $f^{\ast}(d(x_i, x_j))=0$ for all $i,j \in \{ 1, 2, 3\} $ contrary to \eqref{eq_3.5}.

Similarly if $c_1$ and $c_2$ satisfy \eqref{eq_3.3} we can consider an isosceles triangle $X = \{x_1, x_2, x_3\}$ with
\begin{equation}\label{eq_3.6}
d(x_1, x_2) = d(x_2, x_3) = c_2 \quad \textrm{ and} \quad d(x_1, x_3) = c_1.
\end{equation}
Then $(X,d) \in \mathbf{U3}$ and, consequently, we have $(X, f^{\ast} \circ d) \in \mathbf{U}$. The strong triangle inequality implies
$$
f^{\ast}(d(x_1, x_3)) \leqslant \max \{ f^{\ast} (d (x_1, x_2)), f^{\ast} (d(x_2, x_3))\},
$$
which contradicts the second double inequality in~\eqref{eq_3.3}.

Thus \eqref{lem17_eq2} holds. Equality \eqref{lem17_eq1} follows.

Let us prove equality \eqref{lem17_eq3}.

The inclusion $\mathbf{U3} \subseteq {\mathbf{U}}$
and Proposition~\ref{lem_3.1} give us
$
\mathbf{P}_{\mathbf{U3}, \mathbf{U}} \supseteq \mathbf{P}_{\mathbf{U3}}.
$
Thus \eqref{lem17_eq3} holds iff
$$
\mathbf{P}_{\mathbf{U3}} \supseteq \mathbf{P}_{\mathbf{U3}, \mathbf{U}}.
$$
The last inclusion holds iff, for every $f \in \mathbf{P}_{\mathbf{U3}, \mathbf{U}}$ and each $(X,d) \in \mathbf{U3}$, the ultrametric space $(X, f \circ d)$ belongs to ${\mathbf{U3}}$, i.e.,
$$
|X| = 3
$$
holds. It remains to note that the last equality follows from $(X,d) \in {\mathbf{U3}}$ by definition.
\end{proof}

\begin{corollary}\label{cor_20}
The equality
\begin{equation}\label{cor20_eq1}
\mathbf{P}_{ \mathbf{U}}= \mathbf{P}_{\mathbf{U3}}
\end{equation}
holds.
\end{corollary}

It should be mentioned here that Corollary~\ref{cor_20} is not new. See Theorem~9 in~\cite{PT2014FPTA} and Lemma~2.1 in \cite{Ish2021AEaEaaIoU}.

\begin{corollary}\label{cor_21}
The inclusions
\begin{equation}\label{cor21_eq1}
\mathbf{P}_{ \mathbf{U}} \supseteq \mathbf{P}_{\mathbf{CU}, \mathbf{TBU}}, \quad \mathbf{P}_{ \mathbf{U}} \supseteq \mathbf{P}_{\mathbf{CU}}
\quad \textit{and} \quad \mathbf{P}_{ \mathbf{U}} \supseteq \mathbf{P}_{ \mathbf{TBU}}
\end{equation}
hold.
\end{corollary}

\begin{proof}
It follows from \eqref{cor3.2_3.3} and \eqref{lem17_eq1}.
\end{proof}

The following theorem can be considered as a partial generalization of Corollary~\ref{cor_21}.

\begin{theorem}\label{th22}
Let $\mathbf{A}$ and \(\mathbf{B}\) be two classes of ultrametric spaces. If for every $(X,d) \in \mathbf{U3}$ there is $(Y, \rho) \in \mathbf{A}$ such that $(X,d)$ is isometric to a subspace of $(Y, \rho)$, then the inclusion
\begin{equation}\label{th22_eq1}
\mathbf{P}_{\mathbf{A}, \mathbf{B}} \subseteq \mathbf{P}_{\mathbf{U}}
\end{equation}
holds.
\end{theorem}

\begin{proof}
It follows from equality \eqref{cor20_eq1} that \eqref{th22_eq1} holds iff
\begin{equation}\label{th22_pr_eq1}
\mathbf{P}_{\mathbf{A}, \mathbf{B}} \subseteq \mathbf{P}_{\mathbf{U3}}.
\end{equation}
To prove \eqref{th22_pr_eq1} it suffices to show that for every $f \in \mathbf{P}_{\mathbf{A}, \mathbf{B}}$ and each $(X,d) \in \mathbf{U3}$ the space $(X, f \circ d)$ belongs to $\mathbf{U3}$.

Suppose that, for every $(X,d)\in \mathbf{U3}$, there is $(Y, \rho) \in \mathbf{A}$ such that $(X,d)$ is isometric to a subspace of $(Y, \rho)$.

Let us consider arbitrary $f \in \mathbf{P}_{\mathbf{A}, \mathbf{B}}$ and $(X,d) \in \mathbf{U3}$. Then there exist $(Y, \rho) \in \mathbf{A}$ and a three-point subset $Y_1$ of the set $Y$ such that $(X,d)$ and $(Y_1, \rho_1)$ are isometric, where $\rho_1$ is the restriction of $\rho$ on $Y_1 \times Y_1$.

The membership $f \in \mathbf{P}_{\mathbf{A}, \mathbf{B}}$ implies $(Y, f \circ \rho) \in \mathbf{U}$ because $\mathbf{B} \subseteq \mathbf{U}$. Since every subspace of an ultrametric space is ultrametric, we have $(Y_1, \rho_1) \in \mathbf{U3}$. The isometricity of $(X,d)$ and $(Y_1, \rho_1)$ implies the isometricity of $(X, f \circ d)$ and $(Y_1, f \circ \rho_1)$.

Since every metric space isometric to an ultrametric space is itself ultrametric, the membership \((X, f \circ d) \in \mathbf{U3}\) holds. The proof is completed.
\end{proof}

\begin{example}\label{example_19}
Let $\mathbb{R}^{+}$ denote the set $[0, \infty)$. Following Proposition~2 of \cite{DLPS2008TaiA} we define an ultrametric $d^+ \colon \mathbb{R}^+ \times \mathbb{R}^+ \to [0, \infty)$ as
$$
d^+(p,q) := \left\{
\begin{array}{ll}
0 & \quad \hbox{if}\quad p = q, \\
\max \{p,q\} & \quad \hbox{otherwise}.
\end{array}
\right.
$$
Then $(\mathbb{R}^+,d^+)$ belongs $\mathbf{P}_{\mathbf{NUDU}}$ but there are no equilateral triangles in $(\mathbb{R}^+, d^+)$.

Let $\mathbb{R}^{+2}_{0}$ be a subset of $\mathbb{R}^{+} \times \mathbb{R}^{+}$ such that, for every $(s,t) \in \mathbb{R}^+ \times \mathbb{R}^+$,
$$
\left( (s,t)\in \mathbb{R}^{+2}_{0}\right) \Leftrightarrow \left( \min \{ s,t\} = 0 \right).
$$
Write
\begin{equation}\label{e16}
d^+_2 \left( (s_1, t_1), (s_2, t_2)\right)
:= \left\{
\begin{array}{ll}
0 & \quad \hbox{if}\quad (s_1, t_1)= (s_2, t_2), \\
\max \{s_1, t_1, s_2, t_2\} & \quad \hbox{otherwise}.
\end{array}
\right.
\end{equation}
Then $(\mathbb{R}^{+2}_0, d^+_2)$ belongs to $\mathbf{P}_{\mathbf{NUDU}}$ and each $(X,d) \in \mathbf{U3}$ is isometrically embeddable in $(\mathbb{R}^{+2}_0, d^+_2)$. Indeed, let \((X, d)\) belong to \(\mathbf{U3}\) and let
\begin{equation}\label{e17}
D(X) := \{d(x, y) \colon x, y \in X \text{ and } x \neq y\}.
\end{equation}
If the set \(D(X)\) contains exactly two numbers \(d_1\) and \(d_2\), then \(d_1 \neq d_2\) and considering in \((\mathbb{R}^{+2}_0, d^+_2)\) the triangle \(\{(0, 0), (0, d_1), (0, d_2)\}\) we see that this triangle and \((X, d)\) are isometric. For the case when \(D(X)\) is a singleton, \(D(X) = \{d_0\}\), it suffices to consider the equilateral triangle
\[
\{(0, 0), (0, d_0), (d_0, 0)\}.
\]
\end{example}

Now using Theorem~\ref{th22} and Example~\ref{example_19} we obtain.

\begin{corollary}\label{cor23}
The inclusion
$$
\mathbf{P}_{ \mathbf{U}} \supseteq \mathbf{P}_{\mathbf{NUDU}}
$$
holds.
\end{corollary}

\begin{remark}\label{rem24}
Example~\ref{example_19} shows that the space \((\mathbb{R}^{+2}_0, d^+_2)\) is universal for the class \(\mathbf{U3}\) in the sense that, for every \((X, d) \in \mathbf{U3}\), there is an isometric embedding \((X, d) \to (\mathbb{R}^{+2}_0, d^+_2)\). It should be noted that, for every \(R \subset (0, \infty)\), there exists a \(R\)-Urysohn universal ultrametric space with the \(R\)-valued ultrametric. Ishiki~\cite{YS2023, YS2024} seems to success in defining of such spaces for uncountable \(R\). It was noted by reviewer of the paper that \(R\)-Urysohn universal ultrametric space belongs to \(\mathbf{NUDU}\) iff the equality \(\inf R = 0\) holds. Some constructions of universal ultrametric spaces of Urysohn type can be found in \cite{Bog2000VMU, GAO2011TAP, W2021TA}.
\end{remark}

The next example shows that every \((X, d) \in \mathbf{U3}\) can be isometrically embedded into some space \((Y, e) \in \mathbf{TBU} \setminus \mathbf{CU}\).

\begin{example}\label{ex23}
Let \((r_n)_{n \in N}\) be a strictly decreasing sequence of positive real numbers with
\begin{equation}\label{ex23:e1}
\lim_{n \to \infty} r_n = 0,
\end{equation}
and let \(Y = Y(r_n)_{n \in N}\) be a subset of \(\mathbb{R}_0^{+2}\) such that, for every \((s, t) \in \mathbb{R}_0^{+2}\), \((s, t)\) belongs to \(Y\) iff the equality
\begin{equation}\label{ex23:e2}
\max\{s, t\} = r_n
\end{equation}
holds for some \(n \in \mathbb{N}\). Write \(e\) for the restriction of the ultrametric \(d_2^{+}\) on the set \(Y\). Then, using~\eqref{ex23:e1} and Proposition~\ref{pr_7}, we obtain
\begin{equation}\label{ex23:e3}
(Y, e) \in \mathbf{TBU}.
\end{equation}
Limit relation~\eqref{ex23:e1} and equality~\eqref{e16} imply
\[
\lim_{n \to \infty} d_2^{+} ((0, 0), (0, r_n)) = 0.
\]
Moreover, it follows from~\eqref{ex23:e2} that \((0, 0) \notin Y\). Consequently, the ultrametric space \((Y, e)\) is not compact by Proposition~\ref{pr_BW}. Thus, \((Y, e) \in \mathbf{TBU} \setminus \mathbf{CU}\) holds.

Let \((X, d)\) belong to \(\mathbf{U3}\), and let \(D(X)\) be defined by formula~\eqref{e17}. Then there is a strictly increasing sequence \((r_n)_{n \in N}\) of positive reals such that \eqref{ex23:e2} holds and \(D(X)\) is a subset of the range of the sequence \((r_n)_{n \in N}\).

If the set \(D(X)\) contains exactly two members \(r_{n_1}\) and \(r_{n_2}\), then \(n_1 \neq n_2\) and \((X, d)\) is isometric to the triangle \(\{(0, r_{n_1}), (0, r_{n_2}), (0, r_{n_3})\}\) in \((Y, e)\) whenever \(n_3 > \max\{n_1, n_2\}\).

For the case when \(D(X)\) is a singleton, \(D(X) = \{r_{n_0}\}\), it suffices to consider in \((Y, e)\) an equilateral triangle \(\{(0, r_{n_0}), (r_{n_0}, 0), (0, r_{n_1})\}\) whenever \(n_1 > n_0\).
\end{example}

Now using Example~\ref{ex23} and Theorem~\ref{th22} we obtain

\begin{corollary}\label{c26}
If \(\mathbf{A} = \mathbf{TBU} \setminus \mathbf{CU}\) and \(\mathbf{B} = \mathbf{TBU}\), then the inclusion
\[
\mathbf{P}_{\mathbf{U}} \supseteq \mathbf{P}_{\mathbf{A}, \mathbf{B}}
\]
holds.
\end{corollary}

Let us denote by \(\mathbf{P}^0\) the set of all functions \(f \colon [0, \infty) \to [0, \infty)\) satisfying the limit relation
\begin{equation}\label{th18_pr_eq1}
\lim_{x \to 0+} f(x) = f(0),
\end{equation}
i.e., \(f \colon [0, \infty) \to [0, \infty)\) belongs to \(\mathbf{P}^0\) iff the function is continuous at the point \(0\).

The following lemma is equivalent to Lemma~2.2 of paper \cite{Ish2021AEaEaaIoU}.

\begin{lemma}\label{l26}
The equality
\begin{equation}\label{l26:e1}
\mathbf{PT} = \mathbf{P}^0 \cap \mathbf{P}_{\mathbf{U}}
\end{equation}
holds.
\end{lemma}

Lemma~\ref{l26} also admits the following reformulation.

\begin{lemma}\label{lem_27}
A function $f \in \mathbf{P}_{\mathbf{U}}$ belongs to the set $\mathbf{P}_{\mathbf{U}} \setminus \mathbf{PT}$ iff there is $a\in (0, \infty)$ such that
$$
f(s) \geqslant a
$$
for every $s \in (0, \infty)$.
\end{lemma}

\begin{proof}
This statement is equivalent to equality~\eqref{l26:e1}.
\end{proof}

The next theorem can be considered as the main result of the paper.

\begin{theorem}\label{th18}
The equalities
\begin{equation}\label{th18_eq1}
\mathbf{PT}= \mathbf{P}_{\mathbf{CU}}= \mathbf{P}_{\mathbf{TBU}} =\mathbf{P}_{\mathbf{CU}, \mathbf{TBU}}= \mathbf{P}_{\mathbf{NUDU}}
\end{equation}
hold. A function $f \colon [0, \infty) \to [0, \infty)$ belongs to $\mathbf{PT}$ iff $f$ is amenable, increasing and continuous at the point 0.
\end{theorem}

\begin{proof}
The second statement follows from Definition~\ref{def11} and Lemma~\ref{l26}.

Let us prove the first statement. To prove \eqref{th18_eq1} it is enough to check the following equalities
\begin{equation}\label{th18_pr_eq14_1}
\mathbf{P}_{\mathbf{CU}} = \mathbf{PT},
\end{equation}
\begin{equation}\label{th18_pr_eq15}
\mathbf{P}_{\mathbf{CU}, \mathbf{TBU}} = \mathbf{PT},
\end{equation}
\begin{equation}\label{th18_pr_eq14}
\mathbf{P}_{\mathbf{TBU}} = \mathbf{PT},
\end{equation}
and
\begin{equation}\label{th18_pr_eq16}
\mathbf{P}_{\mathbf{NUDU}} = \mathbf{PT}.
\end{equation}

Let us prove equality \eqref{th18_pr_eq14_1}. First of all we note that
\begin{equation}\label{th18_pr_eq16_1}
\mathbf{P}_{\mathbf{CU}} \subseteq \mathbf{P}_{\mathbf{U}},
\end{equation}
holds by Corollary~\ref{cor_21}. Consequently \eqref{th18_pr_eq14_1} holds iff we have
\begin{equation}\label{th18_pr_eq_e1}
\mathbf{PT} \subseteq \mathbf{P}_{\mathbf{CU}}
\end{equation}
and
\begin{equation}\label{th18_pr_eq_e2}
\mathbf{P}_{\mathbf{U}} \setminus \mathbf{PT} \subseteq \mathbf{P}_{\mathbf{U}} \setminus\mathbf{P}_{\mathbf{CU}}.
\end{equation}
To prove \eqref{th18_pr_eq_e1} it suffices to note that, for every $f \in \mathbf{PT} $ and each compact $(X,d) \in \mathbf{U}$, the space $(X, f \circ d)$ is also compact by Definition~\ref{def11}.

Let us prove \eqref{th18_pr_eq_e2}. It suffices to show that
\begin{equation}\label{th18_pr_eq16_2}
f \in \mathbf{P}_{\mathbf{U}} \setminus \mathbf{P}_{\mathbf{CU}}
\end{equation}
holds whenever
\begin{equation}\label{th18_pr_eq16_3}
f \in \mathbf{P}_{\mathbf{U}} \setminus \mathbf{PT}.
\end{equation}
Let \eqref{th18_pr_eq16_3} hold. To prove \eqref{th18_pr_eq16_2} it suffices to show that there is $(X,d) \in \mathbf{CU}$ such that
\begin{equation}\label{th18_pr_eq16_4}
(X, f \circ d) \notin \mathbf{CU}.
\end{equation}
Let us consider a compact ultrametric space $(X,d)$ of infinite cardinality. Then, by Lemma~\ref{lem_27}, for all distinct $x, y \in X$, relation \eqref{th18_pr_eq16_3} implies the inequality
$$
f(d(x,y)) \geqslant a > 0
$$
with some $a\in (0, \infty)$. Consequently we have \eqref{th18_pr_eq16_4} by Proposition~\ref{pr_7}. Inclusion \eqref{th18_pr_eq_e2} follows. The proof of \eqref{th18_pr_eq14_1} is completed.

Let us prove \eqref{th18_pr_eq15}. By Corollary~\ref{cor_21} we have the inclusion
$$
\mathbf{P}_{\mathbf{U}} \supseteq \mathbf{P}_{\mathbf{CU},\mathbf{TBU}}.
$$
Consequently \eqref{th18_pr_eq15} holds iff the inclusion
\begin{equation}\label{th18_pr_eq17_1}
\mathbf{PT} \subseteq \mathbf{P}_{\mathbf{CU}, \mathbf{TBU}}
\end{equation}
and
\begin{equation}\label{th18_pr_eq17_2}
\mathbf{P}_{\mathbf{U}} \setminus \mathbf{PT} \subseteq \mathbf{P}_{\mathbf{U}} \setminus \mathbf{P}_{\mathbf{CU}, \mathbf{TBU}}
\end{equation}
are valid. Let us prove \eqref{th18_pr_eq17_1}. It was noted in the proof of \eqref{th18_pr_eq14_1} that for every $f \in \mathbf{PT}$ and each $(X,d) \in \mathbf{CU}$ we have
\begin{equation}\label{th18_pr_eq17_3}
(X, f \circ d) \in \mathbf{CU}.
\end{equation}
Now \eqref{th18_pr_eq17_3} and \eqref{s2_eq3} imply
$$
(X,f \circ d) \in \mathbf{TBU}.
$$
Consequently every $f \in \mathbf{PT}$ belongs $\mathbf{P}_{\mathbf{CU}, \mathbf{TBU}}$. Inclusion \eqref{th18_pr_eq17_1} follows.

Let us prove \eqref{th18_pr_eq17_2}. It suffices to show that
\begin{equation}\label{th18_pr_eq17_4}
f \in \mathbf{P}_{\mathbf{U}} \setminus \mathbf{P}_{\mathbf{CU}, \mathbf{TBU}}
\end{equation}
whenever \eqref{th18_pr_eq16_3} holds.

Let \eqref{th18_pr_eq16_3} hold. To prove \eqref{th18_pr_eq17_4} it is enough to find $(X,d) \in \mathbf{CU}$ such that
\begin{equation}\label{th18_pr_eq17_5}
(X, f \circ d) \notin \mathbf{TBU}.
\end{equation}
Let us consider a compact ultrametric space of infinite cardinatity. Then, by Lemma~\ref{lem_27}, there is $a \in (0, \infty)$ such that
$$
f(d(x,y)) \geqslant a
$$
for all distinct $x,y \in X$. It implies \eqref{th18_pr_eq17_5} by Proposition~\ref{pr_7}. Inclusion \eqref{th18_pr_eq17_2} follows. The proof of \eqref{th18_pr_eq15} is completed.

Let us prove equality \eqref{th18_pr_eq14}. By Corollary~\ref{cor_21} we have
$$
\mathbf{P}_{\mathbf{TBU}} \subseteq \mathbf{P}_{\mathbf{U}}.
$$
Consequently \eqref{th18_pr_eq14} holds iff
\begin{equation}\label{th18_pr_eq21}
\mathbf{P}_{\mathbf{U}} \setminus \mathbf{PT} \subseteq \mathbf{P}_{\mathbf{U}} \setminus \mathbf{P}_{\mathbf{TBU}}
\end{equation}
and
\begin{equation}\label{th18_pr_eq22}
\mathbf{PT} \subseteq \mathbf{P}_{\mathbf{TBU}}.
\end{equation}
To prove \eqref{th18_pr_eq21} we note that
$
\mathbf{CU} \subseteq \mathbf{TBU}
$
implies
\begin{equation}\label{th18_pr_eq23}
\mathbf{P}_{\mathbf{TBU}} \subseteq \mathbf{P}_{\mathbf{CU}, \mathbf{TBU}}
\end{equation}
by Proposition~\ref{lem_3.1}. Consequently \eqref{th18_pr_eq21} follows from \eqref{th18_pr_eq23} and \eqref{th18_pr_eq17_2}.

Let us prove \eqref{th18_pr_eq22}. Inclusion \eqref{th18_pr_eq22} holds iff, for every $f \in \mathbf{PT}$ and each $(X,d) \in \mathbf{TBU}$, the space $(X, f \circ d)$ is totally bounded.

Let us consider arbitrary $f \in \mathbf{PT}$ and $(X,d) \in \mathbf{TBU}$. Equality \eqref{l26:e1} and limit relation \eqref{th18_pr_eq1} imply that the identical mapping $(X,d) \stackrel{\mathrm{Id}}{\to} (X, f \circ d)$, $\mathrm{Id} (x)= x$ for every $x \in X$, is uniformly continuous and, consequently, this mapping preserves the total boundedness. (See, for example, Theorem 9.2.1 in \cite{Sea2007}). It implies \eqref{th18_pr_eq22}. Equality \eqref{th18_pr_eq14} follows.

Let us prove equality \eqref{th18_pr_eq16}. Using \eqref{l26:e1} we see that equality \eqref{th18_pr_eq16} holds iff
\begin{equation}\label{eq_e1}
\mathbf{P}_{\mathbf{NUDU}} \subseteq \mathbf{PT}
\end{equation}
and
\begin{equation}\label{eq_e2}
\mathbf{P}_{\mathbf{NUDU}} \supseteq \mathbf{P}^{\mathbf{0}} \cap \mathbf{P}_{\mathbf{U}}.
\end{equation}
Let us prove \eqref{eq_e1}.

By Corollary~\ref{cor23} we have
$$
\mathbf{P}_{\mathbf{U}} \supseteq \mathbf{P}_{\mathbf{NUDU}}.
$$
Consequently \eqref{eq_e1} is fails if there is $f \in \mathbf{P}_{\mathbf{NUDU}}$ such that
\begin{equation}\label{eq_e5}
f \in \mathbf{P}_{\mathbf{U}} \setminus \mathbf{PT}.
\end{equation}

Let us consider an arbitrary $f \in \mathbf{P}_{\mathbf{NUDU}}$ satisfying \eqref{eq_e5}. Then, by Lemma~\ref{lem_27}, there exists $a \in (0, \infty)$ such that
\begin{equation}\label{th18_pr_eq6_nudu}
f(t) \geqslant a
\end{equation}
for every $t \in (0, \infty)$. Using \eqref{eq_e5} we obtain
$$
f(d(x,y)) \geqslant a
$$
for every $(X,d) \in \mathbf{NUDU}$ and all distinct $(x,y) \in X$. Now Definition~\ref{def_1.5} implies that $(X, f \circ d)$ is an uniformly discrete ultrametric space, contrary to $f \in \mathbf{P}_{\mathbf{NUDU}}$. Inclusion \eqref{eq_e1} follows.

To prove \eqref{eq_e2} we consider an arbitrary $f \in \mathbf{P}^{\mathbf{0}} \cap \mathbf{P}_{\mathbf{U}}$ and arbitrary $(X,d) \in \mathbf{NUDU}$. By Proposition~\ref{prop_11} there are sequences $(x_n)_{n \in \mathbb{N}}$ and $(y_n)_{n \in \mathbb{N}}$ of points of $X$ such that
\begin{equation}\label{eq_e4}
\lim_{n\to \infty} d(x_n, y_n)=0
\end{equation}
and $x_n \neq y_n$ for every $n \in \mathbb{N}$.
The membership $f \in \mathbf{P}^{\mathbf{0}} \cap \mathbf{P}_{\mathbf{U}}$ implies that $f$ is amenable and continuous at $0$. Now using \eqref{eq_e4} we obtain
$$
\lim_{n\to \infty} f(d(x_n, y_n))=0.
$$
Hence $(X, f \circ d)$ belongs to $\mathbf{NUDU}$ by Proposition~\ref{prop_11}. Thus inclusion \eqref{eq_e2} holds. Equality \eqref{th18_pr_eq16} follows.

The proof is completed.
\end{proof}

In addition to classes $\mathbf{CU},$ $ \mathbf{TBU}$ and $ \mathbf{NUDU}$ there are other classes $\mathbf{X}$ of ultrametric spaces satisfying the equation
$$
\mathbf{P}_{\mathbf{X}} = \mathbf{PT}.
$$

\begin{proposition} \label{prop_27}
Let $\mathbf{A}$ be the class of all totally bounded but non-compact ultrametric spaces,
\begin{equation}\label{s3_eq50}
\mathbf{A} =\mathbf{TBU} \setminus \mathbf{CU}.
\end{equation}
Then the equalities
\begin{equation}\label{s3_eq51}
\mathbf{P}_{\mathbf{A}} = \mathbf{PT},
\end{equation}
\begin{equation}\label{s3_eq51_2}
\mathbf{P}_{\mathbf{A}, \mathbf{TBU}} = \mathbf{PT}
\end{equation}
hold.
\end{proposition}

\begin{proof}
Equality \eqref{s3_eq51} holds iff
\begin{equation}\label{s3_eq52}
\mathbf{PT} \subseteq \mathbf{P}_{\mathbf{A}}
\end{equation}
and
\begin{equation}\label{s3_eq53}
\mathbf{P}_{\mathbf{A}} \subseteq \mathbf{PT}.
\end{equation}
Before proving these inclusions, we note that Theorem~\ref{th22} and Example~\ref{ex23} imply
\begin{equation}\label{s3_eq54}
\mathbf{P}_{\mathbf{A}} \subseteq \mathbf{P}_{\mathbf{U}}.
\end{equation}

Let us prove \eqref{s3_eq52}. We will do it using Theorem~\ref{th18}. Let us consider arbitrary $f \in \mathbf{PT}$ and $(X,d) \in \mathbf{A}$. To prove \eqref{s3_eq52} if suffices to show that
\begin{equation}\label{s3_eq55}
(X, f \circ d) \in \mathbf{TBU}
\end{equation}
and
\begin{equation}\label{s3_eq56}
(X, f \circ d) \notin \mathbf{CU}.
\end{equation}
Membership \eqref{s3_eq55} follows from the relations
$$
\mathbf{PT} = \mathbf{P}_{\mathbf{TBU}}
$$
and
$$
\mathbf{A} \subseteq \mathbf{TBU}
$$
which are valid by \eqref{th18_eq1} and, respectively, by \eqref{s3_eq50}.

To prove \eqref{s3_eq56}, let us assume the opposite
$$
(X. f \circ d) \in \mathbf{CU}.
$$
Equality $\mathbf{PT}=\mathbf{P}_{\mathbf{CU}}$, inclusion \eqref{s3_eq54} and Definition~\ref{def11} show that
$$
(X,d) \in \mathbf{CU},
$$
contrary to $(X,d) \in \mathbf{A}$. Consequently \eqref{s3_eq56} holds. Inclusion \eqref{s3_eq52} follows.

Let us prove \eqref{s3_eq53}. Suppose contrary that there is $f\in \mathbf{P}_{\mathbf{A}} \setminus \mathbf{PT}$. Then using \eqref{s3_eq54} we obtain
$$
f\in \mathbf{P}_{\mathbf{U}} \setminus \mathbf{PT}.
$$
By Lemma~\ref{lem_27}, there is $a \in (0, \infty)$ such that
$$
f(s) \geqslant a
$$
for every $s \in (0, \infty)$. Consequently
$$
f(d(x,y)) \geqslant a
$$
holds for every $(X,d) \in \mathbf{A}$ and all different $x,y \in X$. It implies $(X, f \circ d) \notin \mathbf{A}$ whenever $(X,d) \in \mathbf{A}$, which contradicts $f \in \mathbf{P}_{\mathbf{A}}$. Inclusion \eqref{s3_eq53} follows. The proof of equality \eqref{s3_eq51} is completed.

Let us prove equality \eqref{s3_eq51_2}. By Proposition~\ref{lem_3.1} we have
$$
\mathbf{P}_{\mathbf{A}} \subseteq \mathbf{P}_{\mathbf{A},\mathbf{TBU}}.
$$
The last inclusion and \eqref{s3_eq51} imply
$$
\mathbf{P T} \subseteq \mathbf{P}_{\mathbf{A},\mathbf{TBU}}.
$$
So to prove \eqref{s3_eq51_2} it suffices to show that
\begin{equation}\label{s3_eq57}
\mathbf{P}_{\mathbf{A},\mathbf{TBU}} \subseteq \mathbf{P T}.
\end{equation}
Before proving \eqref{s3_eq57} we note that the inclusion
$$
\mathbf{P}_{\mathbf{A},\mathbf{TBU}} \subseteq \mathbf{P}_{\mathbf{U}}
$$
holds by Corollary~\ref{c26}. Consequently, if \eqref{s3_eq57} not valid, then there is a function
\begin{equation}\label{s3_eq58}
f \in \mathbf{P}_{\mathbf{A},\mathbf{TBU}},
\end{equation}
such that $f \in \mathbf{P}_{\mathbf{U}}$ and $f \notin \mathbf{PT}$.
Now using Lemma~\ref{lem_27} and arguing as in the proof of \eqref{s3_eq51} we can to prove that
$$
f \notin \mathbf{P}_{\mathbf{A},\mathbf{TBU}},
$$
which contradicts \eqref{s3_eq58}. Equality \eqref{s3_eq51_2} follows.
The proof is completed.
\end{proof}

\begin{corollary}\label{cor_28}
Let $\mathbf{A}:=\mathbf{TBU} \setminus \mathbf{CU}$. Then the equality
\begin{equation}\label{cor28_eq1}
\mathbf{P}_{\mathbf{A}, \mathbf{CU}} = \emptyset
\end{equation}
holds.
\end{corollary}

\begin{proof}
Suppose contrary that $ \mathbf{P}_{\mathbf{A}, \mathbf{CU}} \neq \emptyset$ and consider and an arbitrary function
\begin{equation}\label{cor28_eq0}
f \in \mathbf{P}_{\mathbf{A}, \mathbf{CU}}.
\end{equation}
Since $\mathbf{CU} \subseteq \mathbf{TBU}$ holds, Proposition~\ref{lem_3.1} implies that
$$
f \in \mathbf{P}_{\mathbf{A}, \mathbf{TBU}},
$$
and, consequently,
\begin{equation}\label{cor28_eq1-1}
f \in \mathbf{PT}
\end{equation}
by \eqref{s3_eq50}. Let us consider a space
\begin{equation}\label{cor28_eq2}
(X,d) \in \mathbf{A}.
\end{equation}
Then
\begin{equation}\label{cor28_eq3}
(X, f \circ d) \in \mathbf{CU}
\end{equation}
holds by \eqref{cor28_eq0}. Definition \ref{def11}, \eqref{cor28_eq1-1} and \eqref{cor28_eq3} give us
$$
(X,d) \in \mathbf{CU}
$$
but we have
$$
(X,d) \notin \mathbf{CU}
$$
by \eqref{cor28_eq2}. This contradiction shows that \eqref{cor28_eq1} holds.
\end{proof}

\section{Characterization of $\mathbf{P}_{\mathbf{U}}$ via functional equation}

Recall that a triple $(p,q,l)$ of nonnegative real numbers is a \emph{triangle triplet} iff
$$
p \leqslant q+ l , \quad q \leqslant p+ l \quad \textrm{and} \quad l \leqslant p+ q .
$$ 
It was noted in \cite{DM2013} that $(p,q,l) $ is a triangle triplet, iff
\begin{equation}\label{rem16_eq2}
2 \max \{p,q,l\} \leqslant p + q + l.
\end{equation}

An amenable function $f \colon [0, \infty) \to [0, \infty)$ is metric preserving iff
\[
(f(a), f(b), f(c))
\]
is a triangle triplet for every triangle triplet $(a,b,c)$. (See, for example, \cite{BD1981MS} or \cite{Corazza1999}.)

Thus we have the following

\begin{proposition}\label{newprop_22}
An amenable function $f \colon [0, \infty) \to [0, \infty)$ belongs to $\mathbf{P}_{\mathbf{M}}$ iff $f$ preserves inequality \eqref{rem16_eq2}, i.e.,
$$
2 \max \{ f(p), f(q), f(l) \} \leqslant f(p) + f(q) + f(l)
$$
holds whenever we have \eqref{rem16_eq2} for $p, q, l \in [0, \infty)$.
\end{proposition}

\begin{remark}\label{rem_23}
This proposition was first proven in \cite{DM2013} without using the concept of the triangle triplet.
\end{remark}

The following lemma is a reformulation of Lemma~20 of paper \cite{PTAbAppAn2014}.

\begin{lemma}\label{lemma22}
Let $(X,d)$ be a three-point semimetric space with $X=\{x,y,z\}$ and
\begin{equation}\label{lem22_eq1}
d(x,y)=p, \quad d(y,z) =q \quad \textit{and} \quad d(z,x) = l.
\end{equation}
Then the inequalities
\begin{equation}\label{pr_prop22_eq2}
p \leqslant \max \{q,l\}, \quad q \leqslant \max \{p,l\} \quad \textrm{and} \quad l \leqslant \max \{p,q\}
\end{equation}
simultaneously hold iff $(X,d) \in \mathbf{U}$.
\end{lemma}

The next proposition can be considered as an ``ultrametric'' modification of Proposition~\ref{newprop_22}.

\begin{proposition}\label{prop_22}
Let $f \colon [0, \infty) \to [0, \infty)$ be amenable. Then the following statements are equivalent:
\begin{itemize}
\item[$(i)$] $f \in \mathbf{P}_{\mathbf{U}}$;

\item[$(ii)$] The equality
\begin{equation}\label{prop22_eq1}
\begin{gathered}
\min \{ \max \{ f(p), f(q) \}, \max \{ f(q), f(l)\}, \max \{ f (p), f(l)\} \} \\
= \max \{ f(p), f(q), f(l)\}
\end{gathered}
\end{equation}
holds whenever $p, q, l \in [0, \infty)$ and
\begin{equation}\label{prop22_eq2}
\min \{ \max \{ p,q \} , \max \{ q,l\}, \max \{ p, l\} \} = \max \{ p, q, l\} .
\end{equation}
\end{itemize}
\end{proposition}

\begin{proof}
Let $(X,d)$ be a three-point semimetric space satisfying \eqref{lem22_eq1}. First of all we will show that \eqref{prop22_eq2} is valid iff the inequalities in \eqref{pr_prop22_eq2} simultaneously hold.

Assume, without loss of generality, that
\begin{equation}\label{pr_prop22_eq3}
p \leqslant q \leqslant l .
\end{equation}
Then all inequalities in \eqref{pr_prop22_eq2} are satisfied iff $l=q$. Similarly, using \eqref{pr_prop22_eq3} we obtain that $\max \{p, q, l\} = l$ and
\begin{equation}\label{pr_prop22_eq4}
\min \{ \max \{p,q\}, \max \{q,l\}, \max \{p,l\}\} = \max \{p,q\} = q,
\end{equation}
i.e., \eqref{prop22_eq1} holds iff $q=l$.

Thus, \eqref{pr_prop22_eq2} holds iff we have \eqref{prop22_eq2}. The last statement and Lemma~\ref{lemma22} imply that $(X,d)$ is ultrametric iff \eqref{prop22_eq2} is valid.

Similarly we obtain that a semimetric space $(X, f \circ d)$ is ultrametric iff \eqref{prop22_eq1} holds. Thus we have proven the equivalence
\begin{equation}\label{pr_prop22_eq5}
(f \in \mathbf{P}_{\mathbf{U3}}) \Leftrightarrow (ii).
\end{equation}
Now the validity $(i) \Leftrightarrow (ii)$ follows from \eqref{pr_prop22_eq5} and equality~\eqref{cor20_eq1}.
\end{proof}

\begin{corollary}\label{cor_34}
Let $f \colon [0, \infty) \to [0, \infty)$ be amenable. Then the following statements are equivalent:
\begin{itemize}
\item[$(i)$] $f \in \mathbf{P}_{\mathbf{M}} \cap \mathbf{P}_{\mathbf{U}}$;

\item[$(ii)$] For all $p, q, l \in [0, \infty)$, $f$ preserves the inequality
$$
2 \max \{ p,q,l \} \leqslant p+ q+ l .
$$
and the equality
$$
\min \{ \max \{ p,q \} , \max \{ q,l\}, \max \{ l, p\} \} = \max \{ p, q, l\};
$$
\item[$(iii)$] $f$ is subadditive and increasing.
\end{itemize}
\end{corollary}

\begin{proof}
The equivalence $(i) \Leftrightarrow (ii)$ is valid by Proposition~\ref{newprop_22} and Proposition~\ref{prop_22}.

The validity of $(iii) \Rightarrow (i)$ follows from Example~\ref{ex_11} and Theorem~\ref{t2.4}.

Let us prove the validity of $(i) \Rightarrow (iii)$. Let $(i)$ be valid. Then $f$ belongs to $\mathbf{P}_{\mathbf{U}}$ and, consequently, $f$ is increasing by Theorem~\ref{t2.4}. Similarly $(i)$ implies $f \in \mathbf{P}_{\mathbf{M}}$. Now to complete the proof it suffices to note that every metric preserving function is subadditive (see, for example, Proposition~1 in \cite[p.~9]{Dobos1998}).
\end{proof}

The following question seems to be interesting.

\begin{question}
Is there a subclass $\mathbf{X}$ of the class $\mathbf{M}$ such that
$$
\mathbf{P}_{\mathbf{X}} = \mathbf{P}_{\mathbf{M}} \cap \mathbf{P}_{\mathbf{U}}?
$$
\end{question}

\section{Two conjectures}

Recall that a metric space $(X,d)$ is \emph{boundedly compact} or \emph{proper} if each bounded closed subset of $X$ is compact. Let us denote by $\mathbf{A}$ the class of all unbounded boundedly compact ultrametric spaces.

\begin{conjecture}[Prove or disprove]\label{conj_1}
A function $f : [0, \infty) \to [0, \infty)$ belongs to $\mathbf{P}_{\mathbf{A}}$ iff $f \in \mathbf{P}_{\mathbf{CU}}$ and
\begin{equation}\label{conj_eq1}
\lim\limits_{t \to +\infty} f(t) = + \infty.
\end{equation}
\end{conjecture}

A metric space $(X,d)$ is \emph{locally finite} if $|S| < \infty$ holds for every bounded subset $S$ of $X$. Let us denote by $\mathbf{B}$ the class of all uniformly discrete unbounded locally finite ultrametric spaces.

\begin{conjecture}[Prove or disprove]\label{conj_2}
A function $f : [0, \infty) \to [0, \infty)$ belongs to $\mathbf{P}_{\mathbf{B}}$ iff $f \in \mathbf{P}_{\mathbf{U}}$ and \eqref{conj_eq1} holds.
\end{conjecture}

\section*{Acknowledgment}

The author would like to thank the anonymous referee whose nontrivial remarks strongly helped in preparation of the final version of the article.

\section*{Funding information}

The author was supported by the Academy of Finland (Project ``Labeled trees and totally bounded ultrametric spaces'').

\end{document}